\documentclass[12pt]{article}
\usepackage{amsmath, amssymb, amsthm}
\usepackage{algorithm}
\usepackage{algpseudocode}
\usepackage{graphicx}
\usepackage{float} 
\usepackage{tikz}
\usetikzlibrary{calc,through,intersections, arrows, decorations.markings, matrix,shapes,decorations}
\usepackage{braids}
\usepackage{thmtools}
\usepackage{thm-restate}
\usetikzlibrary{braids}

\theoremstyle{plain}
\newtheorem{theorem}{Theorem}[section]
\newtheorem{lemma}[theorem]{Lemma}
\newtheorem{proposition}[theorem]{Proposition}

\theoremstyle{definition}
\newtheorem{definition}[theorem]{Definition}

\newtheorem*{remark}{Remark}

\title{Every Plat Presentation Admits a Positive Bounded Braid Representative}
\author{Seth Hovland}

\begin{document}

\maketitle

\begin{abstract}
Classifying link types via their $n$-bridge positions requires navigating vast Hilden double coset classes  $H_{2n} \backslash B_{2n} / H_{2n}$. We provide a restriction on this by proving that every Hilden double coset admits a positive braid representative whose Dehornoy minimum is explicitly bounded. The proof relies on the key observation that $\Delta \in H_{2n}$, which allows us to utilize the Garside left-greedy normal form to remove all powers of $\Delta$ from any given braid representative. We provide an explicit upper bound on its Dehornoy minimum in terms of the Garside length of the initial presentation. This bounded representative restricts the algebraic space of plat presentations, providing a potential computational tool for studying bridge isotopy classes and offering a constructive step toward algebraic approaches to the Bridge Finiteness Conjecture.
\end{abstract}

\section{Introduction}

In classical knot theory, the relationship between a braid presentation and the geometric properties of its corresponding link heavily depends on the choice of closure. Under the standard braid closure, enforcing algebraic positivity on a braid representative severely restricts the topology of the link; closed positive braids yield only positive links, a highly specialized class known to be fibered and deeply constrained in their cross-cap numbers and signatures. 

In contrast, the equivalence classes governing plat presentations characterized topologically as bridge isotopy classes and algebraically as Hilden double cosets $H_{2n} \backslash B_{2n} / H_{2n}$ (See \cite{hovland2024bridgepositionsplatpresentations}) exhibit a very different paradigm: \emph{positivity in plat presentations is a universal feature shared by all link types.}

Recently, Ozawa \cite{Ozawa2026} proved the existence of canonical representatives of Hilden double coset classes by defining a Dehornoy-type ordering on the set of Hilden double coset classes. In this paper, we establish a parallel result using a strictly constructive approach: we prove that every Hilden double coset admits a minimal positive braid representative and provide an explicit upper bound on its Dehornoy minimum. Additionally, it was shown in \cite{Ozawa2026} that the Bridge Finiteness Conjecture would follow from a topological bound on canonical representatives of Hilden double cosets. We show that such a bound would follow from a corresponding bound on the Garside length of a plat presentation.
We summarize our result as Theorem \ref{thm:main} below:
\begin{restatable}{theorem}{MainTheorem}\label{thm:main}
Given a braid word $\beta \in B_{2n}$ whose Garside decomposition contains $p$ permutation braids, its Hilden double coset $[\beta]$ admits a positive representative whose Dehornoy minimum is bounded above by $p(2n-1).$
\end{restatable}

\section{Background}
\subsection{Positive Braids and Plat Closures}

Let $B_n$ denote the braid group on $n$ strands with standard Artin generators $\sigma_1, \dots, \sigma_{n-1}$. We denote by $B_n^+$ the monoid of \emph{positive braids}, consisting of words expressed entirely in the generators $\sigma_i$ without inverses. As established by Garside \cite{Garside1969}, the natural embedding of $B_n^+$ into $B_n$ provides a foundation for canonical factorizations and algorithmic solutions to the word problem.

We obtain a link from an even stranded braid $\beta \in B_{2n}$ by taking its \emph{plat closure}, formed by identifying the bottom of the $i$-th strand to the $(i+1)$-st strand for all odd $i$, and performing the identical pairwise identification at the top. It is a classical result that every link type can be represented as the plat closure of some braid.

Note that local isotopies near the top and bottom bridges do not alter the topological link type, therefore, plat presentations are highly non-unique. Birman showed in \cite{Birman1976OnTS} that two braids yield isotopic plat closures if and only if, after a finite number of stabilizations, they lie in the same double coset of the \emph{Hilden subgroup}.

\subsection{The Garside Element and the Hilden Subgroup}\label{sec:garsidehildendefns}

Let $\Delta \in B_n$ denote the braid
\[
\Delta := (\sigma_1\sigma_2 \cdots \sigma_{n-1})
          (\sigma_1\sigma_2 \cdots \sigma_{n-2})
          \cdots
          (\sigma_1\sigma_2)(\sigma_1).
\]
This braid is called the \emph{Garside element} (or \emph{fundamental element}). It plays a central role in the algebraic structure of the braid group.  For instance, its square generates the center of $B_n$ and therefore commutes with every braid. Moreover, $\Delta$ itself ``almost commutes'' with elements of $B_n$ by the relation $\Delta\sigma_i= \sigma_{n-i}\Delta.$  We focus particularly on the use of $\Delta$ in the construction of normal forms for braids. 

\begin{proposition}[Garside {\cite{Garside1969}}]
Every braid $\beta \in B_n$ admits a factorization
\[
\beta = \Delta^m P_0,
\]
where $m \in \mathbb{Z}$ and $P_0$ is a positive braid containing no factor equal to $\Delta$. Among all positive representatives of $P_0$, there is a unique lexicographically minimal word, which we also denote by $P_0$.
\end{proposition}

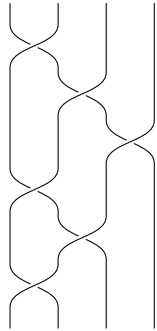
\begin{figure}[h!]
    \centering
    \begin{tikzpicture}
     \pic[braid/.cd,
     number of strands=4,
     line width=1.5pt,
     name prefix=braid, 
     height=.25in,
     width=0.25in,
     style=thin] at (0,0) {braid={s_1 s_2 s_1 s_3 s_2 s_1}};
    \end{tikzpicture}
    \caption{The Garside element $\Delta$ in $B_4$.}
\end{figure}

This expression of a braid as a power of $\Delta$ followed by a positive braid not containing $\Delta$ as a factor is called a \emph{Garside decomposition}.  Occasionally, we will use the fact that the positive braid $P_0$ admits a decomposition into $P_0=p_1p_2\cdots p_p$ where each $p_i$ is a \emph{permutation braid}. We refer the reader to \cite{BLMCGs} and \cite{WordprocessinginGroups} for further background, including applications to the word and conjugacy problems in braid groups.

We now turn to the Hilden Subgroup of the braid group which governs the equivalence of plat closures.

\begin{definition}[Hilden Subgroup {\cite{Hilden1975GeneratorsFT}}]\label{Defn:Hildensubgroup}
Let $H_{2n} \le B_{2n}$ denote the subgroup consisting of all braids that preserve plat closures. Equivalently, a braid $\gamma \in B_{2n}$ lies in $H_{2n}$ if and only if, for every $\beta \in B_{2n}$, the plat closures of $\gamma\beta$ and $\beta\gamma$ represent the same link type as the plat closure of $\beta$.
\end{definition}

Hilden showed that $H_{2n}$ is generated by the finite set
\[
\{\sigma_1,\; \sigma_2\sigma_1^2\sigma_2,\;
\sigma_{2i}\sigma_{2i-1}\sigma_{2i+1}\sigma_{2i}
\mid 1 \le i \le n-1\},
\]
where $\sigma_i$ are the standard Artin generators of $B_{2n}$.

A presentation for $H_{2n}$, including a complete set of relators, was later given by Tawn \cite{Tawn_2008}. Geometrically, elements of $H_{2n}$ correspond to ambient isotopies supported near the top or bottom bridges of a plat presentation, which return each bridge to its original height. As a consequence, multiplication on the left or right by elements of $H_{2n}$ updates the underlying braid word but leaves the link type of the plat closure invariant.

When comparing plat presentations, we will frequently refer to these isotopies as \emph{double coset moves}, which algebraically correspond to multiplication on the left or right by elements of the Hilden subgroup.

\begin{remark}
Definition~\ref{Defn:Hildensubgroup} emphasizes the knot-theoretic interpretation of the Hilden subgroup to build geometric intuition. For a strictly algebraic characterization via the mapping class group actions used in computation, see standard reference texts like \cite{BirmanPlatsLinksBraids}.
\end{remark}

\begin{figure}[H]
    \centering
    \includegraphics[width=0.75\linewidth]{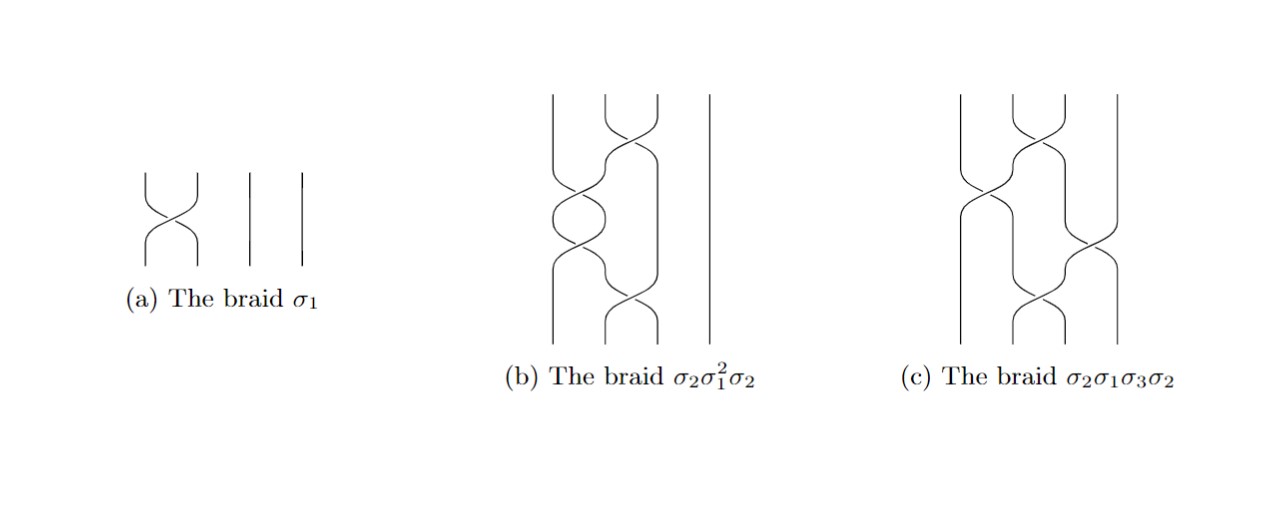}
    \caption{Generators of the Hilden subgroup $H_4$.}
    \label{fig:generatorsH4}
\end{figure}

\subsection{The Hilden Subgroup Membership Problem}

The main goal of this section is to provide a geometric method of identifying Hilden elements.

\begin{proposition}{\cite{SSpolytimealgorthims,Tawn2009}}
The membership problem for the Hilden subgroup is decidable.
\end{proposition}

In \cite{SSpolytimealgorthims} Schleimer showed that the membership problem for the mapping class group of a handlebody, viewed as a subgroup of the mapping class group of its boundary, is solvable in polynomial time. Building on this work, Tawn \cite{Tawn2009} adapted Schleimer’s argument to the braid-theoretic setting and proved that the membership problem for the Hilden subgroup is likewise decidable in polynomial time. We present the geometric structure of this criterion below, which we will use to show the Garside element is in the Hilden Subgroup.

\subsubsection{The Mapping Class Perspective}

To solve the membership problem algorithmically, we identify the braid group $B_{2n}$ with the mapping class group of the $2n$-punctured disk, denoted $X_{2n}$, fixing the boundary pointwise \cite{BLMCGs}. A braid $\beta \in B_{2n}$ induces an automorphism of the fundamental group,
\[
\beta_* \colon \pi_1(X_{2n}) \to \pi_1(X_{2n}).
\]
Since $\pi_1(X_{2n})$ is isomorphic to the free group $F_{2n}$, we track the action of $\beta$ algebraically via its induced action on a free generating set adapted to the Hilden subgroup, as shown in Figure~\ref{fig:nicegens}.

\begin{figure}[H]
    \centering
    \includegraphics[width=0.6\linewidth]{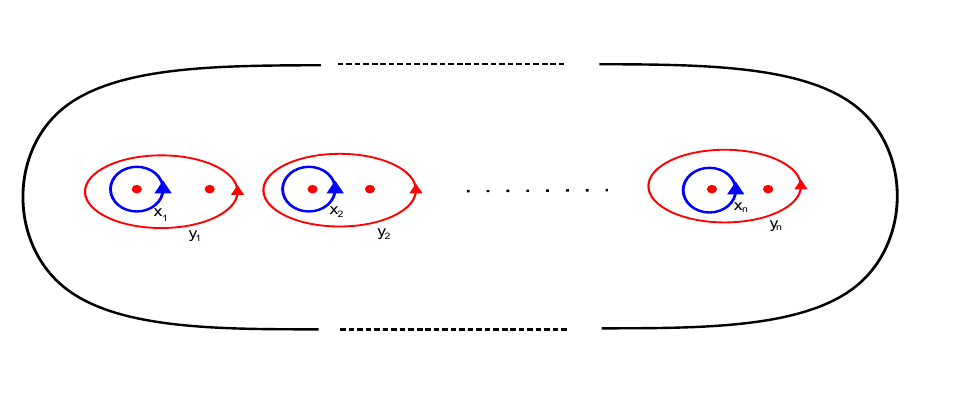}
    \caption{A generating set for $\pi_1(X_{2n})$ chosen for the Hilden subgroup.}
    \label{fig:nicegens}
\end{figure}

Consider the inclusion of $X_{2n}$ into a 3-ball containing $n$ unknotted arcs, denoted $B^3_+ \setminus \mathcal{A}$. Under this filling, the loops labeled $y_i$ bound disks in the complement and become null-homotopic:
\[
\iota: X_{2n} \hookrightarrow B^3_+ \setminus \mathcal{A}, \qquad
\iota_*: \langle x_1, \dots, x_n, y_1, \dots, y_n \rangle 
\longrightarrow \langle x_1, \dots, x_n \rangle.
\]

\begin{figure}[H]
    \centering
    \includegraphics[width=0.6\linewidth]{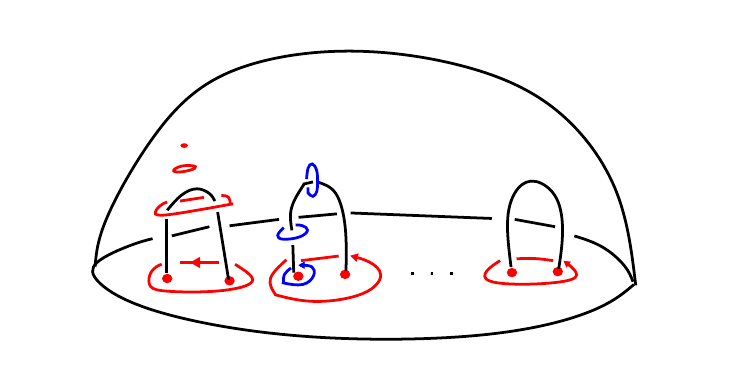}
    \caption{The generators $y_i$ bound disks when $X_{2n}$ is viewed inside the 3-ball.}
    \label{fig:gensin3ball}
\end{figure}

Geometrically, a braid $\beta \in B_{2n}$ lies in the Hilden subgroup if and only if its plat closure forms an $n$-component unlink, meaning the closure admits $n$ disjoint splitting spheres. This geometric property can be tracked efficiently by evaluating the image of the loops $y_i$ under the braid action.

\begin{figure}[H]
    \centering
    \includegraphics[width=0.7\linewidth]{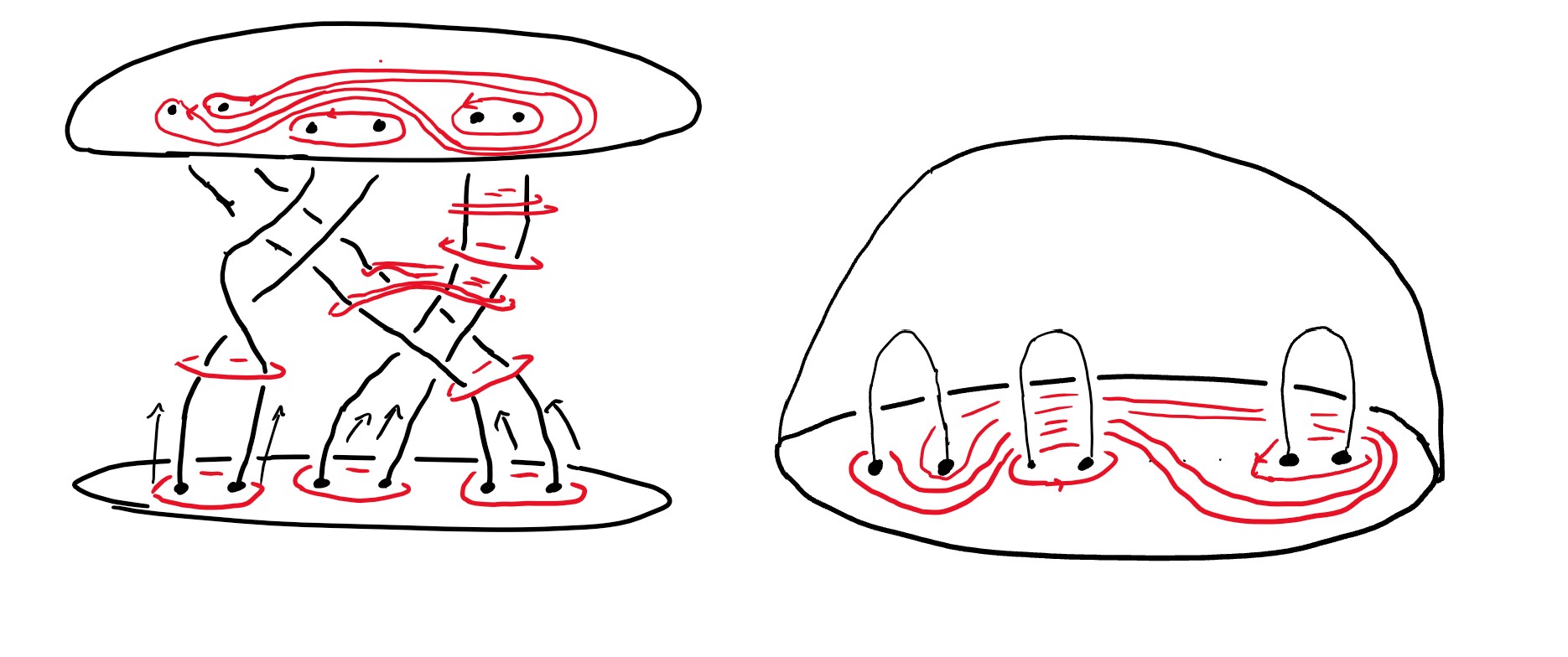}
    \caption{A geometric view of Tawn's approach to testing Hilden subgroup membership.}
    \label{fig:tawndoublecosetidea}
\end{figure}

The free group word $\beta_*(y_i)$ can be computed in time polynomial in the word length of $\beta$ \cite{SSpolytimealgorthims}. We conclude with Tawn’s formal algebraic criterion:

\begin{theorem}{\cite{Tawn2009}}\label{Thm:hildenmembership}
A braid $\beta \in B_{2n}$ lies in the Hilden group $H_{2n}$ if and only if, for each $i = 1,\dots,n$, the induced element $\beta_*(y_i) = 1$ in $\pi_1(B_{+}^3\setminus\mathcal{A}).$
\end{theorem}
\begin{proof}
Suppose $\beta \in H_{2n}$. By definition, $\beta$ preserves the set of arcs $\mathcal{A}$ up to isotopy in $B^3_+$. Therefore, any loop $y_i \subset X_{2n}$ that bounds a disk in $B^3_+ \setminus \mathcal{A}$ is sent by $\beta$ to a loop that also bounds a disk in the filled complement, satisfying $\beta_*(y_i) = 1 \in \pi_1(B^3_+ \setminus \mathcal{A})$.

Conversely, suppose $\beta \in B_{2n}$ maps each loop $y_i$ to a null-homotopic word. By Dehn's Lemma, we can pick disjoint disks $D_i$ and $D_i'$ in $B^3_+ \setminus \mathcal{A}$ bounding $y_i$ and $\beta(y_i)$ respectively. By Alexander's Trick, the boundary homeomorphism extends across the disks. The collection of disks separates $B^3_+$ into $n$ balls each containing a single trivial arc. Because $\beta$ is the identity on $\partial X_{2n}$, it extends to a global homeomorphism of the 3-ball mapping $\mathcal{A}$ to itself setwise, matching the definition of the Hilden subgroup.
\end{proof}

\section{The Positive Cone of a Hilden Double Coset}

Let $\beta \in B_{2n}$ and denote by $[\beta] = H_{2n}\beta H_{2n}$ its Hilden double coset class. We define the \emph{positive cone} $[\beta]^+$ of this class to be the subset of all algebraically positive braid words representing it:
\[
[\beta]^+ := \{\gamma \in [\beta] \mid \gamma \in B_{2n}^+\}.
\]

\begin{lemma}\label{lemma:positivecone}
The positive cone $[\beta]^+$ is non-empty for any braid $\beta \in B_{2n}$.
\end{lemma}
\begin{proof}
By Garside's theorem, any braid can be factored as $\beta = \Delta^m P_0$, where $P_0 \in B_{2n}^+$ is a positive braid word. The critical observation is that the Garside element $\Delta$ itself belongs to the Hilden subgroup $H_{2n}$. Geometrically, the Garside element $\Delta$ represents a positive half-twist of the entire $2n$-punctured disk. Because the plat closure pairs punctures $\{1,2\}, \{3,4\}, \dots \{2n-1, 2n\}$, this global half-twist maps the bridge pair $\{1,2\}$ onto the bridge pair $\{2n-1, 2n\}$ and the bridge pair $\{3,4\}$ onto the bridge pair $\{2n-3, 2n-2\}$ etc., preserving the set of unknotted arcs $\mathcal{A}$. Therefore, for $\Delta\in B_{2n}$ we have 
$\Delta_*(y_i)=y_{n-i+1}$ for $1\leq i \leq n.$  The inclusion of each of these elements into $\pi_1(B_+^3\setminus\mathcal{A})$ is trivial so that by Theorem~\ref{Thm:hildenmembership}  $\Delta \in H_{2n}.$
Alternatively, one may verify via double coset isotopies that the plat closure of $\Delta$ yields an $n$-component unlink (see Figure~\ref{fig:garsidehildenmoves}).

Because $\Delta \in H_{2n}$, any power $\Delta^m$ can be absorbed by the left action of the Hilden subgroup. Consequently, $\beta = \Delta^m P_0 \in H_{2n}P_0 H_{2n}$, proving that $P_0 \in [\beta]^+$.
\end{proof}

\begin{figure}[H]
    \centering
    \includegraphics[width=0.6\linewidth]{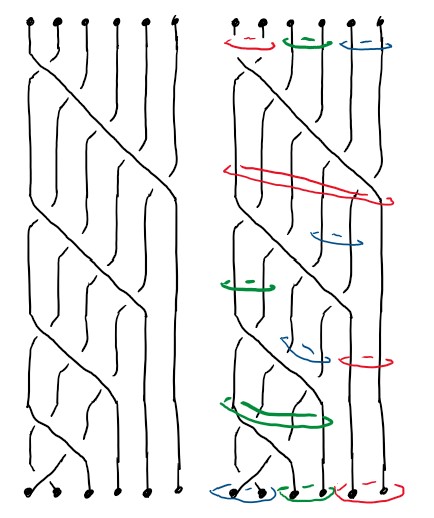}
    \caption{Tracking the action of the Garside element on the $y_i$ boundary loops shows the action of $\Delta$ on each $y_i.$}
    \label{fig:garsideoncurves}
\end{figure}

\begin{figure}[H]
    \centering
    \includegraphics[width=0.75\linewidth]{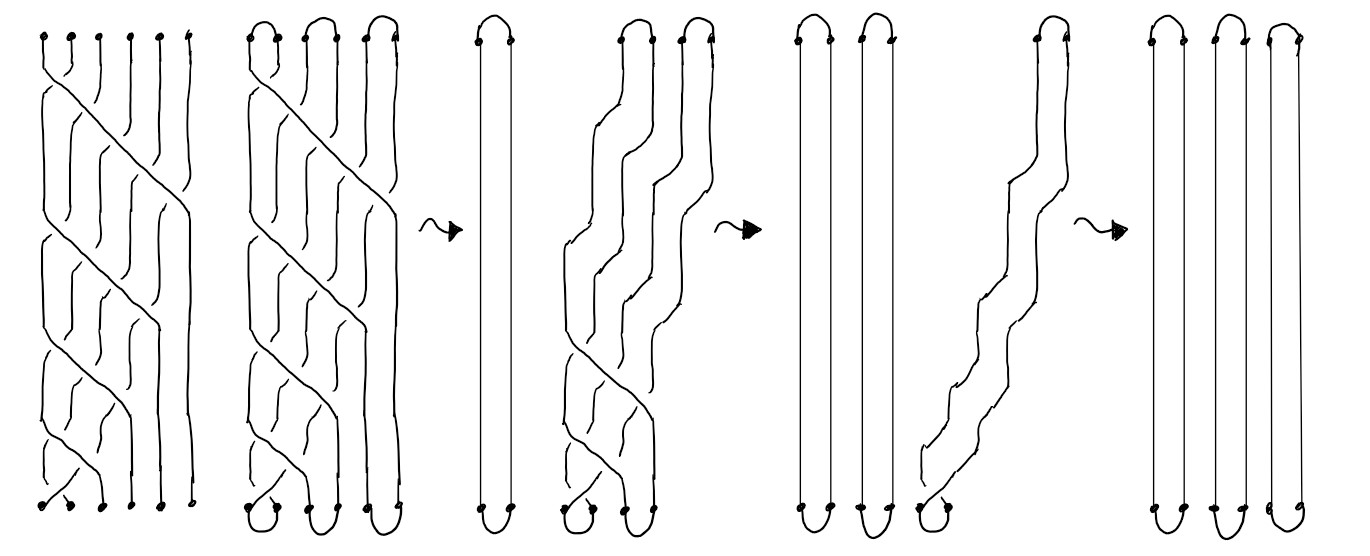}
    \caption{Isotoping the plat closure of $\Delta$ into a split unlink using double coset moves.}
    \label{fig:garsidehildenmoves}
\end{figure}

 Since the set of positive braid words is well-ordered lexicographically, each double coset contains a unique minimal representative.

\begin{definition}
Using the Artin generators for $B_{2n},$ the \emph{lexicographical minimum} of $[\beta]^+$ is the braid word $\gamma_0 \in [\beta]^+$ that is lexicographically smaller than all other elements in $[\beta]^+.$
\end{definition}
The minimal element $\gamma_0 \in [\beta]^+$ necessarily minimizes word length within the positive cone. We denote the word length of a braid by $\text{Len}(\beta).$ The element $\gamma_0$ serves as a unique canonical representative for the double coset class $[\beta]$.

We note that the existence of canonical representatives for plat presentation classes has also recently been established by Ozawa \cite{Ozawa2026}, who leveraged a Dehornoy-type total ordering on the double cosets. Our approach offers a complementary, computationally oriented perspective: by utilizing the Garside decomposition, we restrict the search space entirely to the positive cone, where the well-ordering of positive braids trivially yields a canonical form. In the next section, we show that this specific construction allows us to explicitly bound the complexity of this representative.

\begin{remark}
It is important to distinguish between the \emph{existence} of a positive representative in $[\beta]^+$ and the algorithmic computability of the canonical lexicographical minimum $\gamma_0$. Navigating the double coset requires multiplication by Hilden generators, which generally introduces negative crossings. Restoring a positive braid word after such a double coset move requires passing back through the Garside normal form, which may temporarily reduce the power of $\Delta$ and pull the representative out of the positive cone. Consequently, while the positive cone is guaranteed to be non-empty and bounded, navigating it directly remains a non-trivial algorithmic challenge.
\end{remark}

\section{The Dehornoy Order of Braids and Plat Closures}

We now consider how these positive plat representatives behave with respect to the Dehornoy ordering of the braid group \cite{Dehornoy}.

\subsection{The Dehornoy Ordering of Braid Groups}

\begin{definition}
A braid word $\beta \in B_n$ is \emph{$D$-positive} if it contains at least one occurrence of $\sigma_i$ for some index $i$ ($1 \le i \le n-1$), and no occurrences of $\sigma_1^{\pm 1}, \dots, \sigma_{i-1}^{\pm 1}$ or $\sigma_i^{-1}$. A braid is \emph{$D$-negative} if $\beta^{-1}$ is $D$-positive. For $\beta_1, \beta_2 \in B_n$, we write $\beta_1 <_D \beta_2$ if the braid $\beta_1^{-1}\beta_2$ is $D$-positive.
\end{definition}

Crucially, every non-trivial braid is either $D$-positive or $D$-negative, but never both. This total order is right-invariant: if $\beta_1 <_D \beta_2$, then $\beta_1\alpha <_D \beta_2\alpha$ for any $\alpha \in B_n$. We track intervals bounded by powers of the Garside element using the standard notation:
\[
(\Delta^a,\Delta^b):=\{\beta \in B_n \mid \Delta^a <_D \beta <_D \Delta^b\}.
\]
We recall a useful framing property established by Malyutin and Netsvetaev:
\begin{lemma}[Malyutin \& Netsvetaev \cite{Malyutin_Netsvetaev}]\label{lemma:malyutin}
Suppose a braid $\beta \in B_n$ is represented by a word containing $r$ occurrences of $\sigma_1$ and $s$ occurrences of $\sigma_1^{-1}$ with $r+s>0.$ Then $\beta \in (\Delta^{-2s},\Delta^{2r}).$
\end{lemma}

Symmetric intervals coming from positive and negative powers of $\Delta$ are widely used to partition $B_n$ via the Dehornoy floor.

\begin{definition}
The \emph{Dehornoy floor} of a braid $\beta \in B_n,$ denoted $[\beta]_D,$ is the non-negative integer defined by:
\[
\lbrack\beta]_D = \min\{m \in \mathbb{Z}_{\geq 0} \mid \beta \in (\Delta^{-2(m-1)},\Delta^{2(m+1)})\}.
\]
\end{definition}

A rephrasing of Lemma~\ref{lemma:malyutin} yields an immediate upper bound on this floor:
\begin{lemma}
Suppose a braid $\beta \in B_n$ is represented by a word containing $r$ occurrences of $\sigma_1$ and $s$ occurrences of $\sigma_1^{-1}$ with $r+s>0.$ Then $[\beta]_D < \max\{r,s\}.$
\end{lemma}

\subsection{Dehornoy Bounds on Plat Presentation Classes}\label{sec:dehornoy-bounds}

As proven in Lemma~\ref{lemma:positivecone}, the fact that $\Delta \in H_{2n}$ means that for any braid $\beta = \Delta^m P_0$ in Garside normal form, the factor $\Delta^m$ can be cleanly removed via left-multiplication by the Hilden subgroup without changing the link type of the plat closure. Thus, the plat closures $\hat{\beta}$ and $\hat{P}_0$ share an identical Hilden double coset class.

This behavior contrasts sharply with the closed braid setting, where algebraic positivity restricts the link to a narrow, specialized topological class (positive links). This structural shift motivates a tailored adaptation of the Dehornoy floor metric specifically for plat presentations, allowing us to explicitly calculate the bound asserted in Theorem~\ref{thm:main}.

\begin{definition}
Given an $n$-bridge plat $B$ defined by the braid word $\beta$, the \emph{Dehornoy minimum} $B_D$ is the non-negative integer defined by:
\[
B_D := \min\{m \in \mathbb{Z}_{\geq 0} \mid \gamma \in [\beta]  \text{ and } \gamma \in (1,\Delta^m)\}.
\]
\end{definition}

The Dehornoy minimum $B_D$ is an invariant of the Hilden double coset of $\beta.$ It quantifies how  ``tangled up'' the plat position is even after removing the noise from the negative crossings and powers of the Garside element. 

We may now establish a bound on the complexity of possible positive representatives for the plat $B$.

\begin{lemma}\label{lemma:dehornoybound}
    Given a braid word $\beta \in B_{2n}$ whose Garside decomposition contains $p$ permutation braids, its Hilden double coset $[\beta]$ admits a positive representative with Dehornoy minimum bounded below by $1$ and above by $p(2n-1).$
\end{lemma}
\begin{proof}
    We obtain a product of positive permutation braids $P_1P_2\cdots P_p$ by finding the Garside decomposition of $\beta.$ The maximal number of occurrences of any generator $\sigma_i$ inside a single permutation braid is bounded by the complexity of the Garside element $\Delta \in B_{2n}$, where $\sigma_1$ occurs exactly $2n-1$ times. It follows by Lemma~\ref{lemma:malyutin} that the Dehornoy minimum is less than $p(2n-1).$
\end{proof}

With both the geometric structure of the positive cone established and the Dehornoy bound calculated, the proof of our primary result follows immediately.
\MainTheorem*

\begin{proof}
By Lemma~\ref{lemma:positivecone}, absorbing any negative powers of the Garside element into the Hilden subgroup yields an entirely positive braid representative $P \in B_{2n}^+$ within the same double coset class. By the preceding lemma, the Dehornoy minimum $B_D$ of this representative is strictly bounded above by $p(2n-1)$, completing the proof.
\end{proof}

\begin{remark}
The upper and lower bounds established in Theorem~\ref{thm:main} and Lemma~\ref{lemma:dehornoybound} are fundamentally algebraic, relying on the Garside length $p$ of the initial braid word $\beta$. Recently, Ozawa proved that the Bridge Finiteness Conjecture would follow from a topological bound on the complexity of canonical representatives of Hilden double cosets (Proposition 6.3). Our theorem reduces this problem to obtaining a topological bound on the Garside length p of an arbitrary plat presentation. Indeed, if p can be bounded in terms of the link type and bridge number, then our Dehornoy bound immediately yields such a topological bound on the canonical representative, and hence the Bridge Finiteness Conjecture. We leave the exploration of this bound on $p$ to future work.
\end{remark}

\bibliographystyle{plain}
\bibliography{Sources}

\vspace{15mm}

\newcommand{\Addresses}{
  \footnotesize
  Seth Hovland, \\
  Department of Mathematics,\\
  University of Minnesota Duluth,\\
  Duluth, MN 55812, USA \\
  \textit{E-mail address}: \texttt{hovl0070@d.umn.edu} 
}

\Addresses

\end{document}